\documentclass{amsart}

\usepackage{amsmath}
\usepackage{graphicx}
\usepackage{url}
\usepackage{mathrsfs}

\newtheorem {prop} {Proposition}

\newtheorem {conj} {Conjecture}
\newtheorem* {coup} {Move}

\def\S{\mathbb{S}}
\def\T{\mathbb{T}}
\def\P{\mathbb{P}}

\title{Sprouts game on compact surfaces}
\author{Julien Lemoine - Simon Viennot}

\begin{document}

\begin{abstract}
Sprouts is a two-player topological game, invented in 1967 by Michael Paterson and John Conway. The game starts with $p$ spots drawn on a sheet of paper, and lasts at most $3p-1$ moves: the player who makes the last move wins.\par
Sprouts is a very intricate game and the best known manual analysis only achieved to find a winning strategy up to $p=7$ spots. Recent computer analysis reached up to $p=32$.\par
The standard game is played on a plane, or equivalently on a sphere. In this article, we generalize and study the game on any compact surface.\par
First, we describe the possible moves on a compact surface, and the way to implement them in a program. Then, we show that we only need to consider a finite number of surfaces to analyze the game with $p$ spots on any compact surface: if we take a surface with a genus greater than some limit genus, then the game on this surface is equivalent to the game on some smaller surface. Finally, with computer calculation, we observe that the winning player on orientable surfaces seems to be always the same one as on a plane, whereas there are significant differences on non-orientable surfaces.
\end{abstract}

\maketitle

\section{Introduction}

\textit{Sprouts} is a two-player game, which needs only a sheet of paper and a pen to enjoy. Rules are extremely simple, and can be found for example in Martin Gardner's article of 1967 \cite{mg67} (when the game was invented).\par
The player who makes the last move wins, and since the number of moves of a game with $p$ spots is finite (at most $3p-1$), the game cannot end in a draw. It implies that one of the players has a winning strategy.\par
Most of the interest of Sprouts is to find the player having a winning strategy, but it is made difficult by the game's complexity. The first manual analysis only achieved to find a winning strategy up to $p=6$ spot, and it necessitated a lot of reasoning to consider all the possible cases. The first program of Sprouts enabled Applegate, Jacobson and Sleator \cite{ajs91} in 1991, to extend this analysis up to $p=11$ spots, and to formulate the following conjecture:

\begin{conj}
The first player has a winning strategy in games starting with $p$ spots, if and only if $p$ is equal to 3, 4 or 5 modulo 6.
\end{conj}

Later computation in 2007 \cite{lv07} proved this conjecture to be true at least up to $p=32$ spots.\par
In this paper, we generalize Sprouts to other surfaces than the simple plane of a sheet of paper. First of all, let us remark that the game is equivalent on the sphere and on the plane. Indeed, if we consider a final Sprouts position after the game was played on a plane, we can ``wrap'' the plane around a sphere (an adequate mapping is the stereographic projection). This means that the exact same game could have been played on the sphere. Conversely, a similar argument shows that any game played on the sphere could have been played on the plane.

\begin{figure}[ht]
\centering
\includegraphics[scale=0.5]{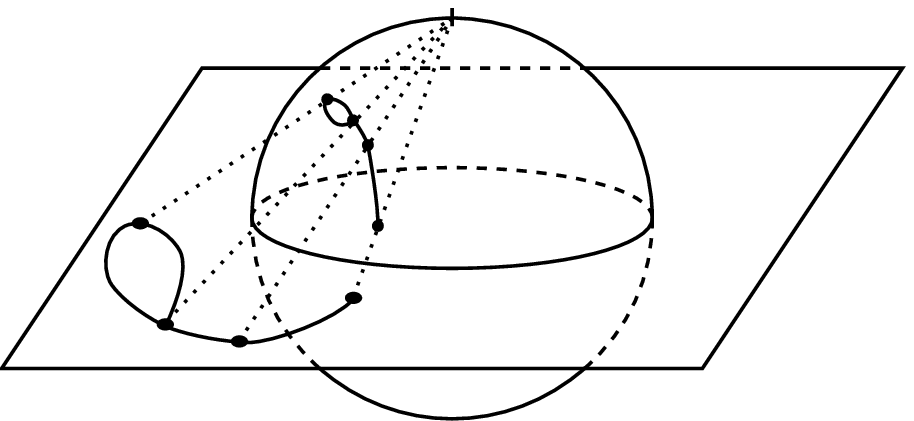}
\caption{Stereographic projection }
\end{figure}

Sprouts natural playground is a surface: if we tried to play Sprouts in a volume, it would be always possible to connect two given spots with a line, and any game would end in exactly $3p-1$ moves. Interestingly, we can consider other surfaces than the sphere and in this paper, we generalize Sprouts to any compact surface. We describe first how to play Sprouts on compact surfaces, by listing all the possible moves, then we give our first results on the winning strategy, obtained by adapting our program of \cite{lv07}. We applied the same programming methods as in \cite{lv07}, simply generalizing some functions to the case of compact surfaces.

\section{Basic notions on compact surfaces}

\subsection{Sprouts}

We call \emph{position} a Sprouts game at a given time: this is a graph $\mathscr{G}$ embedded in a surface $\mathscr{S}$. We call \emph{region} the union set of a connected component of $\mathscr{S}-\mathscr{G}$ and of all the parts of $\mathscr{G}$ touching it. Because of the rule forbidding to cross an existent line, a move is always done inside a given region. After a move, the region remains unchanged, or breaks into two new regions, so that the total number of regions can only increase during a game.\par
Inside a region, we call \emph{boundaries} the connected component of the part of the graph $\mathscr{G}$ touching the considered region. For example, the position of figure \ref{simple} has been obtained from an initial position with 3 spots, $A$, $B$ and $C$. $A$ has been linked to $B$, creating $D$, then $A$ to $D$ creating $E$. There are two regions in this position: the first one has only one boundary, which can be written ``$ADBDE$'', and the second one has two boundaries, ``$ADE$'' and ``$C$''.\par
A \emph{life} is a remaining possibility for a point to be linked to another point: initially, a vertex has 3 lives, and if $k$ lines ($k\leq 3$) are linked to a vertex, it remains $3-k$ lives.

\begin{figure}[ht]
\centering
 \includegraphics[scale=0.5,angle=-90]{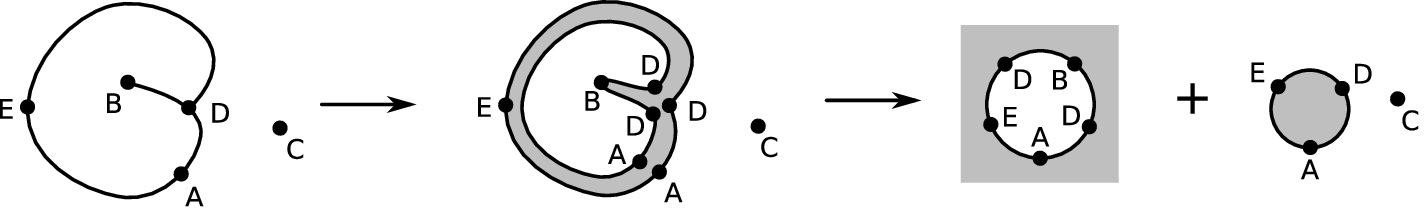}
\caption{A simple position}
\label{simple}
\end{figure}

\subsection{Orientable surfaces}

The sphere will be denoted $\S$, the torus $\T$. The connected sum of two surfaces $\mathscr{S}_1$ and $\mathscr{S}_2$ will be denoted $\mathscr{S}_1 \sharp \mathscr{S}_2$. The connected sum of two tori (or \emph{torus with two holes}) will be denoted $\T^2$, and, more generally, the connected sum of $n$ tori will be denoted $\T^n$, for any $n\geq 1$. To simplify some results, we also use $\T^0=\S$.

\subsection{Non-orientable surfaces}

The real projective plane will be denoted $\P$. The connected sum of two projective planes (the well-known \emph{Klein bottle}) will be denoted $\P^2$, and more generally, the connected sum of $n$ projective planes will be denoted $\P^n$, for any $n\geq 1$.

\subsection{Classification of compact surfaces}
\label{class}
The classification of closed surfaces states that any compact boundaryless surface is homeomorphic to $\T^n$ (for some $n\geq 0$) if the surface is orientable, or to $\P^n$ (for some $n\geq 1$) if the surface is non-orientable \cite{fw99}. Moreover, the connected sum of two surfaces is done in the following manner:

\begin{center}
\begin{tabular}{ rcp{2cm}l }
$\T^m \sharp \T^n$ & $=$ & $\T^{m+n}$ & $(m,n\geq 0)$ \tabularnewline
$\P^m \sharp \P^n$ & $=$ & $\P^{m+n}$ & $(m,n\geq 1)$ \tabularnewline
$\P^m \sharp \T^n$ & $=$ & $\P^{m+2n}$ & $(m\geq 1,\ n\geq 0)$ \tabularnewline
\end{tabular}
\end{center}

\subsection{Euler characteristic}

Let $\chi(\mathscr{S})$ denote the Euler characteristic of the surface $\mathscr{S}$. This characteristic has the following values on compact surfaces:

\begin{center}
\begin{tabular}{ rcp{2cm}l }
$\chi(\T^n)$ & $=$ & $2-2n$ & $(n\geq 0)$ \tabularnewline
$\chi(\P^n)$ & $=$ & $2-n$ & $(n\geq 1)$ \tabularnewline
\end{tabular}
\end{center}

$n$ is called the \emph{genus} of the surface.

If the surface $\mathscr{S}$ is not connected, its Euler characteristic is the sum of Euler characteristics of each connected components. If a surface $\mathscr{S}$ is homeomorphic to a compact surface $\mathscr{S}_C$ with $b$ boundaries, then $\chi(\mathscr{S})=\chi(\mathscr{S}_C)-b$. For example, a surface $\mathscr{S}$ with two connected components, one homeomorphic to a torus with 2 holes, and one homeomorphic to a sphere with 3 boundaries, has the following Euler characteristic: $\chi(\mathscr{S})=(-2)+(2-3)=-3$.

``Gluing'' or ``cutting'' a surface along a boundary doesn't change the Euler characteristic of this surface. For example, when gluing two surfaces $\mathscr{S}_1$ and $\mathscr{S}_2$ to construct their connected sum, we add a boundary to $\mathscr{S}_1$ and a boundary to $\mathscr{S}_2$, then we glue along these boundaries. We obtain: $\chi(\mathscr{S}_1 \sharp \mathscr{S}_2)=(\chi(\mathscr{S}_1)-1)+(\chi(\mathscr{S}_2)-1)$.

\subsection{Plane representation of compact surfaces}

If it is easy to play Sprouts on a plane with a simple sheet of paper, it seems, on the other hand, uncomfortable to use a three-dimensional doughnut to explore the game on the torus. Fortunately, a mathematical tool allows us to represent any compact surface on the plane: the fundamental polygon. The fundamental polygon of the projective plane and of the torus are represented on figure \ref{ppt}. On each polygon, sides with matching labels are pairwise identified, so that the arrows point in the same direction. A single dotted line has been drawn and shows the meaning of each polygon.

\begin{figure}[ht]
\centering
 \includegraphics[scale=0.5]{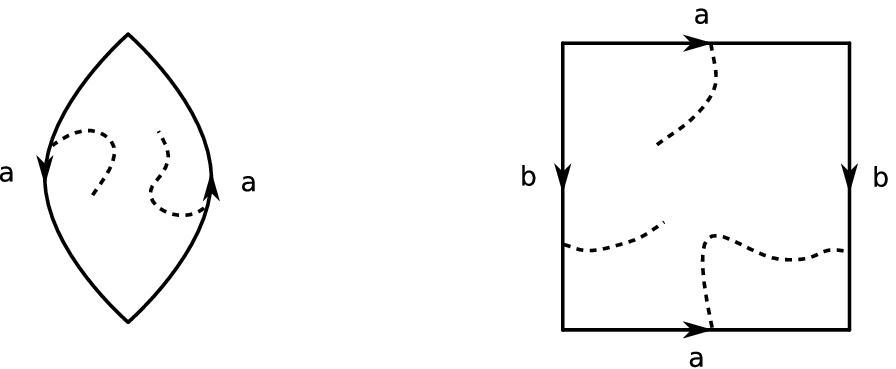}
\caption{Fundamental polygons of the projective plane and the torus}
\label{ppt}
\end{figure}

The represented surface can be effectively constructed by cutting a sheet of paper in the shape of the polygon and gluing sides with matching labels, paying attention to orientation (the arrows must point in the same direction). In the case of non-orientable surfaces, the real construction is however impossible in the three-dimensional space...

A fundamental polygon can be written with a string of characters. Beginning at any vertex, the string is the list of the side's names when going around the whole polygon. For example, the fundamental polygon of $\P$ given in figure \ref{ppt} is written $aa$, and these of $\T$ is  $aba^{-1}b^{-1}$ (beginning at the upper left, and going around clockwise). The letter is $a$ when the side's arrow points in movement direction, and $a^{-1}$ elsewhere.

The fundamental polygon of all other surfaces can be constructed with the following rule: the fundamental polygon of the connected sum of two surfaces is obtained by simply enumerating the two fundamental polygons of the connected surfaces. For example, a fundamental polygon of $\P^2$ is  $aabb$, and a fundamental polygon of $\T^2$ is $aba^{-1}b^{-1}cdc^{-1}d^{-1}$.

The fundamental polygon is not unique for a given surface. For example, $aabb$ is a fundamental polygon for the Klein Bottle $\P^2$, but $abab^{-1}$ is another valid one.

\subsection{Surface related to a region}

When playing Sprouts on a compact surface, all regions are homeomorphic to a compact surface with boundaries. There is in fact a complete equivalence between boundaries considered in the theory of surfaces and boundaries considered in the theory of Sprouts. We call the \textit{surface related to a region} the boundaryless compact surface homeomorphic to the region without its boundaries. It is important to note that a Sprouts game is not affected by the number of boundaries of the surface. For example, we can add boundaries with only dead points, without interfering with the game.

\section{Kind of moves}
\label{types}

In this section, we describe the kind of moves possible on compact surfaces. When playing on the plane, there are only two kinds of moves: the connection of two different boundaries, or the connection of a boundary to itself, breaking the current region into two new regions. But when we generalize to compact surfaces, there are a whole lot of other possible moves.

\subsection{List of possible moves}

We assume that we are playing a move in a region $\mathscr{R}$ of a Sprouts game. The surface related to $\mathscr{R}$ is denoted $\mathscr{S}$.

\begin{coup}[kind I]
Link between two different boundaries.
\end{coup}

This move is just a generalization of the move on a plane. It results in the merging of the two boundaries, and the region $\mathscr{R}$ remains homeomorphic to the same compact surface.

\begin{coup}[kind II]
Link of a boundary to itself.
\end{coup}

When we link a boundary to itself, it creates a loop $\mathscr{L}$ which modifies the topological properties of the region $\mathscr{R}$. In the case of the plane, the region breaks into two new regions, but there are other possibilities on compact surfaces. Let us detail them.

\begin{coup}[kind II.A]
$\mathscr{R}-\mathscr{L}$ has two connected components.
\end{coup}

This move is somewhat similar to the move on the plane. It breaks the region $\mathscr{R}$ into two compact surfaces, whose connected sum is homeomorphic to $\mathscr{R}$.
By inverting the relations of paragraph \ref{class}, we obtain:

\begin{center}
\begin{tabular}{ rcp{2.5cm}p{5cm}l }
$\T^n$ & $\rightarrow$ & $\T^k+\T^{(n-k)}$ & $(n\geq 0,\ 0\leq k\leq n)$ & $(a)$ \tabularnewline
$\P^n$ & $\rightarrow$ & $\P^k+\P^{(n-k)}$ & $(n\geq 2,\ 0 < k < n)$ & $(b)$ \tabularnewline
$\P^n$ & $\rightarrow$ & $\T^k+\P^{(n-2k)}$ & $(n\geq 1,\ 0\leq k < \frac{n}{2})$ & $(c)$ \tabularnewline
\end{tabular}
\end{center}

\begin{figure}[ht]
\centering
 \includegraphics[scale=0.5,angle=-90]{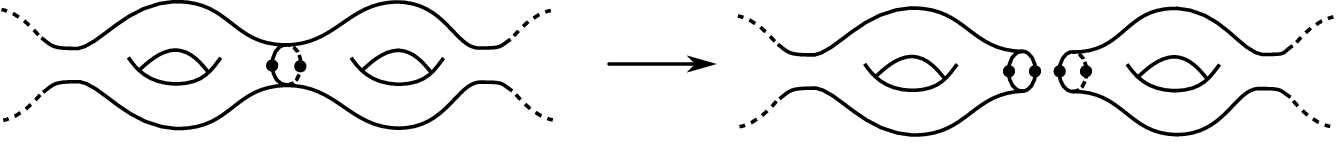}
\caption{Move of kind II.A.a}
\end{figure}

\begin{coup}[kind II.B]
$\mathscr{R}-\mathscr{L}$ has only one connected component.
\end{coup}

The surface related to $\mathscr{R}-\mathscr{L}$ will be denoted $\mathscr{S'}$. There are two sub-cases:

\begin{coup}[kind II.B.1]
Cutting $\mathscr{R}$ along $\mathscr{L}$ creates two boundaries.
\end{coup}

We have the relation $\chi(\mathscr{S})=\chi(\mathscr{S'})-2$, and since cutting an orientable region doesn't change its orientability, we only need to consider the Euler characteristic to list the possible cases:

\begin{center}
\begin{tabular}{ rcp{1.5cm}p{4cm}l }
$\T^n$ & $\rightarrow$ & $\T^{(n-1)}$ & $(n\geq 1)$ & $(a)$ \tabularnewline
$\P^n$ & $\rightarrow$ & $\P^{(n-2)}$ & $(n\geq 3)$ & $(b)$ \tabularnewline
$\P^n$ & $\rightarrow$ & $\T^{\frac{n-2}{2}}$ & $(n=2k, k\geq 1)$ & $(c)$ \tabularnewline
\end{tabular}
\end{center}

This move is achieved by ``breaking'' the handle of a torus in the first two cases, or the neck of a Klein bottle in the last two ones (the second case can be considered in two ways).

\begin{figure}[ht]
\centering
 \includegraphics[scale=0.5]{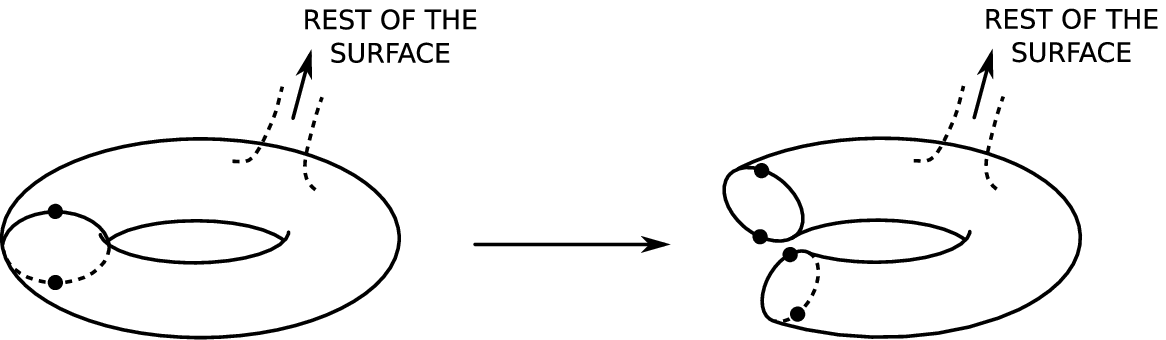}
\caption{Move of kind II.B.1.(a) or II.B.1.(b)}
\end{figure}

\begin{figure}[ht]
\centering
 \includegraphics[scale=0.5,angle=-90]{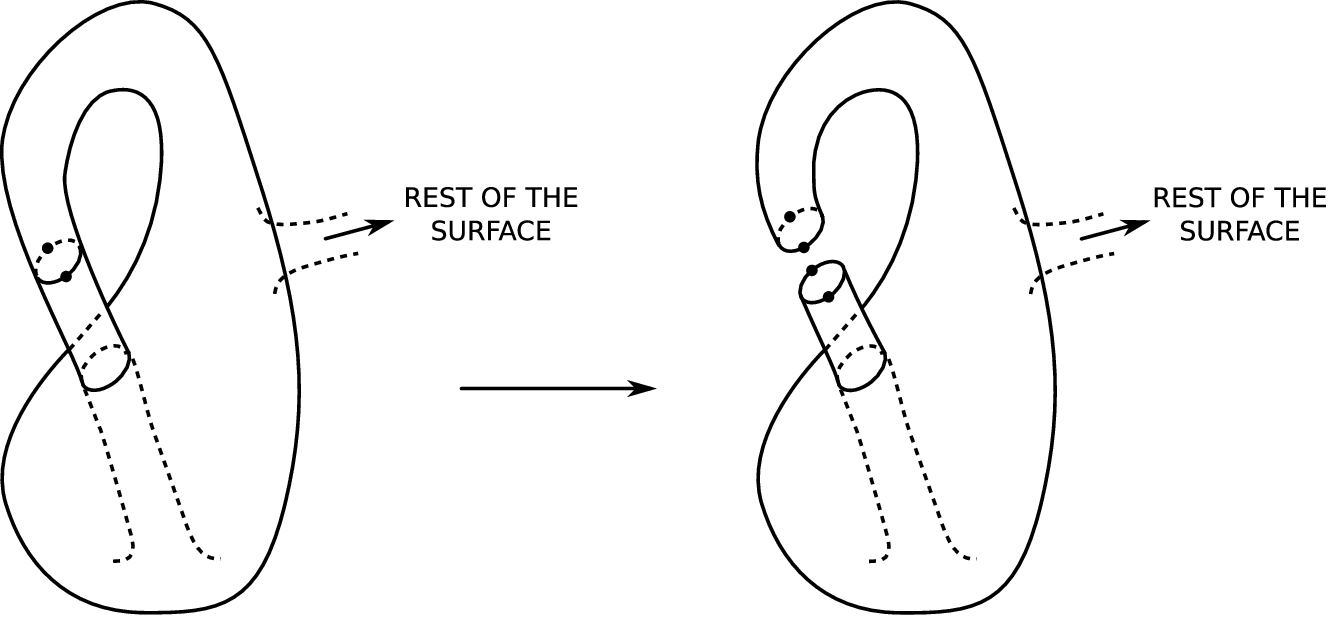}
\caption{Move of kind II.B.1.(b) or II.B.1.(c)}
\end{figure}

\begin{coup}[kind II.B.2]
Cutting $\mathscr{R}$ along $\mathscr{L}$ creates only one boundary.
\end{coup}

We obtain the relation $\chi(\mathscr{S})=\chi(\mathscr{S'})-1$, which leads to:

\begin{center}
\begin{tabular}{ rcp{1.5cm}p{4cm}l }
$\P^n$ & $\rightarrow$ & $\P^{(n-1)}$ & $(n\geq 2)$ & $(a)$ \tabularnewline
$\P^n$ & $\rightarrow$ & $\T^{\frac{n-1}{2}}$ & $(n=2k+1, k\geq 0)$ & $(b)$ \tabularnewline
\end{tabular}
\end{center}

This move is achieved by breaking a a projective plane. It is the move that the first player should play to win the 2-spot game on the projective plane: it links a spot to itself as on figure \ref{planproj}, and then the first player can force the game to end in 5 moves.

\begin{figure}[ht]
\centering
 \includegraphics[scale=0.5]{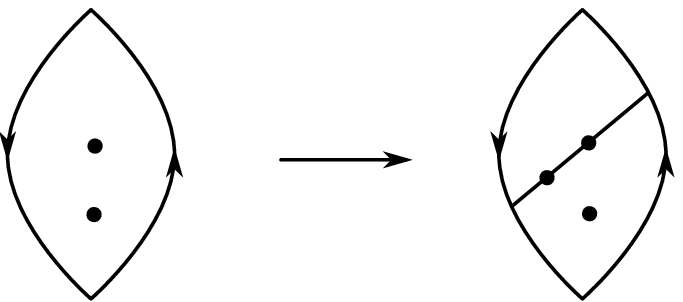}
\caption{Move of kind II.B.2.(b) to win the 2-spot game on the projective plane}
\label{planproj}
\end{figure}

\subsection{Particular cases}

In the case of the plane, we use only moves of kind I and of kind II.A.(a) with $n=k=0$.

For the game on orientable surfaces, we use moves of kind I, kind II.A.(a), and kind II.B.1.(a). Let us remark that only this last kind of move, II.B.1.(a), which is achieved by ``breaking a handle'', can lead to a game different from what is possible on the plane.

\section{Programming}

\subsection{Consequences of orientability}

In respect to orientability, regions related to an orientable surface follow the same kind of rule as the plane. When enumerating the vertices of a boundary, we need to turn around it in the way chosen for the region (clockwise or counter-clockwise), so that we turn around all boundaries in the same way.

For a region related to a non-orientable surface, there is no more concept of orientability. When enumerating the vertices of a boundary, we can freely choose the way we turn around it, either clockwise or counter-clockwise. Figure \ref{inorient} shows that on the projective plane, the same boundary can be written equally \textit{abc} or \textit{acb}.

\begin{figure}[ht]
\centering
\includegraphics[scale=0.5]{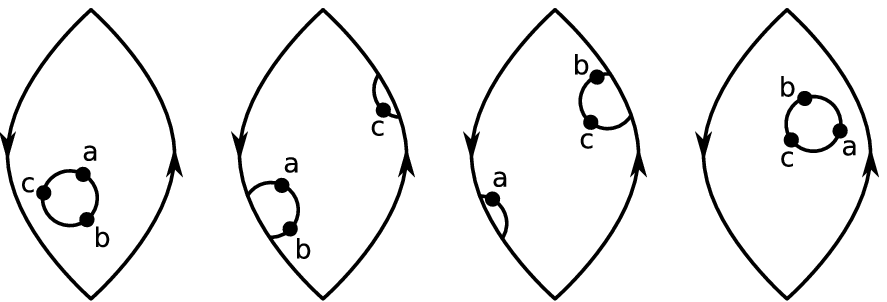}
\caption{Consequence of the projective plane non-orientability}
\label{inorient}
\end{figure}

An orientable region will remain orientable all through the game, but on the contrary non-orientable regions can lead either to non-orientable or to orientable regions.  It is possible to create an orientable region from a non-orientable one with the moves of kind II.A.(c), II.B.1.(c) or II.B.2.(b). In that case, the player can freely choose the orientation of each boundary. Indeed, each set of orientation choices corresponds to a possible move. For example, figure \ref{choixorient} shows a move of kind II.B.2.(b). Orientation of boundaries depends on the way the line is drawn.

\begin{figure}[ht]
\centering
\includegraphics[scale=0.5]{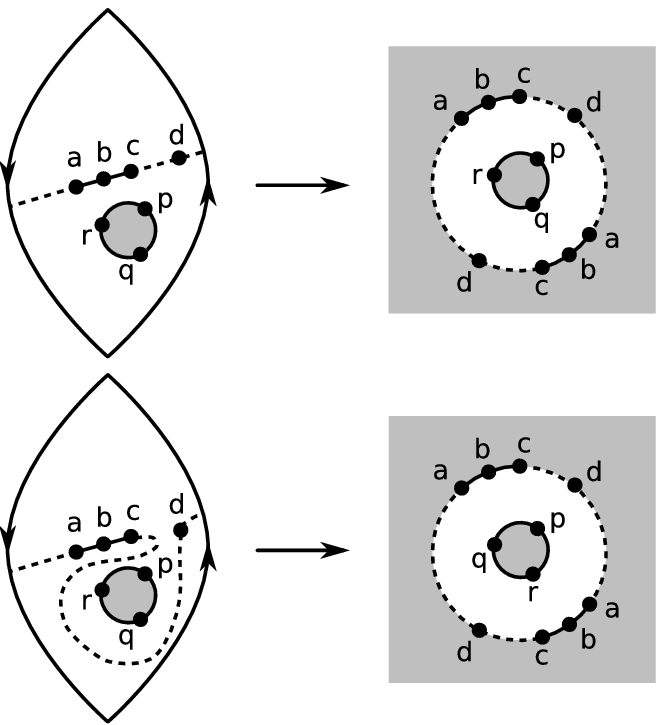}
\caption{Transition from a non-orientable surface to an orientable one}
\label{choixorient}
\end{figure}

\subsection{Positions representation}

Sprouts positions are represented in the program by list of vertices. A detailed description in the case of the plane will be found in \cite{lv07}. We describe here the changes needed to generalize this representation to compact surfaces.

The main change is the need to remember for each region of the position what is its related compact surface. For that, we simply add to the representation of each region an integer, corresponding to its genus. Orientable regions are represented with a positive integer, and non-orientable ones with a negative integer.

Moreover, on compact surfaces, a new kind of one-life vertex appears. On the plane, a one-life vertex is of two kinds only: either it belongs to a single region and therefore to a single boundary, or it belongs to two different regions and therefore two different boundaries. But in the case of compact surfaces, moves of kind II.B.1 create one-life vertices belonging to two boundaries, yet still in the same region. Although it is not necessary to distinguish between these three kinds of one-life vertex in order to program Sprouts, it allowed us to optimize some of the functions related to canonization, and it helps checking the correctness of some algorithms.

Finally, let us detail a tricky problem occurring with one-life vertices that appear two times in the same boundary. When we play on the plane, it is possible to simplify the representation by using brackets instead of letters to name these vertices\footnote{This idea is from Dan Hoey.}. For example, on figure \ref{parenth} (left), we can represent the boundary in two equivalent manners: $abcdedcafgf$ or $(b((e)))(g)$.

\begin{figure}[ht]
\centering
\includegraphics[scale=0.5]{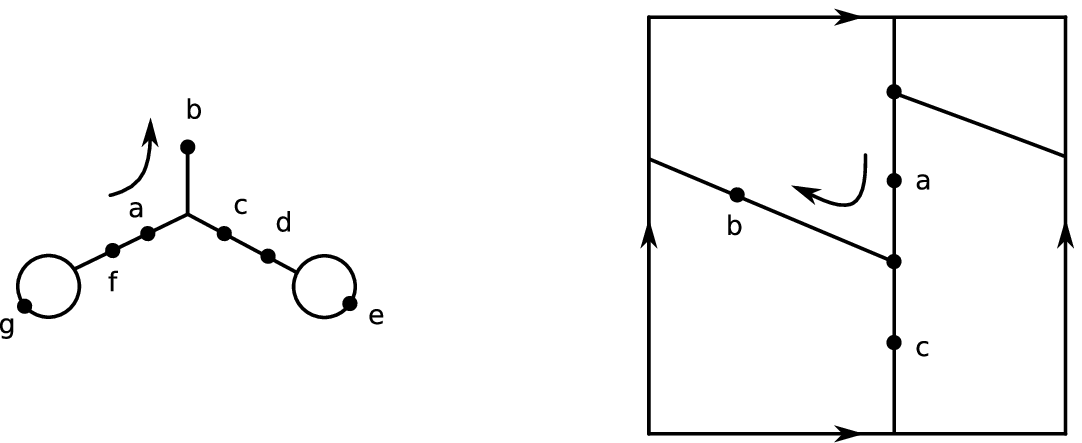}
\caption{Parenthesizable and non-parenthesizable positions}
\label{parenth}
\end{figure}

However, this way of using brackets is not possible when playing on other surfaces than the plane. For example, the position on figure \ref{parenth} (right), occurs in a 2-spot game on the torus, after three moves. If we turn around the boundary in the way indicated by the arrow, the representation becomes \textit{abcabc}, which wouldn't be correctly represented with brackets.

\subsection{Merging boundaries}

When we link two boundaries, the result is a new boundary obtained by merging the two initial ones. Let us detail this merging process for each kind of move described in section \ref{types}.

\begin{itemize}
\renewcommand{\labelitemi}{*}
\item move of kind I: Let suppose that $(a_1,...,a_r)$ with $r\geq 2$ and $(b_1,...,b_s)$ with $s\geq 2$ are two boundaries, and that we link $a_i$ to $b_j$, creating a new vertex $c$. Then, the result of the merging is the new boudary \\$(a_i,...a_r,a_1,...,a_i,c,b_j,...,b_s,b_1,...,b_j,c)$ if we play in an orientable region, and could also be $(a_i,...a_r,a_1,...,a_i,c,b_j,...,b_1,b_s,...,b_j,c)$ if we play in a non-orientable one. Note: in the particular case where the boundary $(a_1,...,a_r)$ has only one vertex $a_1$, we have $a_i,...a_r,a_1,...,a_i=a_1$.
\item move of kind II.A, II.B.1.(a) or II.B.1.(b): If we link $a_i$ to $a_j$ ($i\leq j$) in the boundary $(a_1,...,a_r)$, creating a new vertex $c$, then we obtain two boundaries: $(a_j,...a_r,a_1,...,a_i,c)$ and $(a_i,...,a_j,c)$.
\item move of kind II.B.1.(c): If we link $a_i$ to $a_j$ ($i\leq j$) in the boundary $(a_1,...,a_r)$, creating a new vertex $c$, then we obtain two boundaries: \\$(a_j,...a_r,a_1,...,a_i,c)$ and $(a_j,...,a_i,c)$.
\item move of kind II.B.2.: If we link $a_i$ to $a_j$ ($i\leq j$) in the boundary $(a_1,...,a_r)$, creating a new vertex $c$, then we obtain only one boundary: \\$(a_j,...a_r,a_1,...,a_i,c,a_j,...,a_i,c)$.
\end{itemize}

\section{Canonical game trees}

We introduce in this section a useful concept to define an equivalence between positions leading to ``the same game''. This will help us in the next section to prove some theoretical results on compact surfaces.

We call \textit{game tree obtained from a position $\mathscr{P}$} the tree where vertices are the positions obtained when playing moves from $\mathscr{P}$, and where two positions $\mathscr{P}_1$ and $\mathscr{P}_2$ are linked by an edge if  $\mathscr{P}_2$ is obtained with a move from $\mathscr{P}_1$.

We define recursively an equivalence relation between the game trees: two game trees $\mathscr{A}_1$ and $\mathscr{A}_2$, obtained respectively from $\mathscr{P}_1$ and $\mathscr{P}_2$, are equivalent if the set of equivalence classes of trees obtained from children of $\mathscr{P}_1$ is equal to the set of equivalence classes from trees otained from children of $\mathscr{P}_2$. We call \textit{canonical tree} a tree in the equivalence class with a minimal number of vertices.

\begin{figure}[ht]
\centering
\includegraphics[scale=0.5]{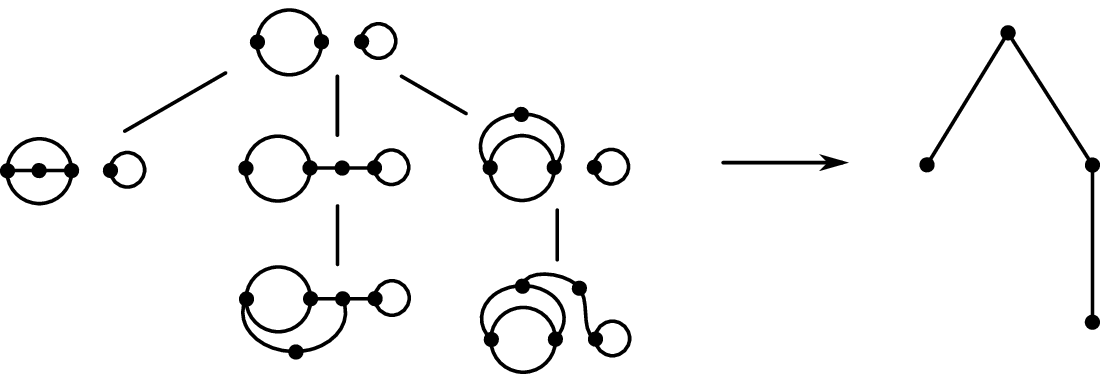}
\caption{Game tree obtained from a given position and corresponding canonical tree}
\label{arbres}
\end{figure}

Figure \ref{arbres} shows on the left a game tree obtained from a given Sprouts position. The two branches on the right lead to similar games, so that we can merge these two branches into a single one in the canonical game tree. The game tree, as well as the canonical game tree, are of height 2, because the longest game possible ended in 2 moves.

\begin{figure}[ht]
\centering
\includegraphics[scale=0.5,angle=-90]{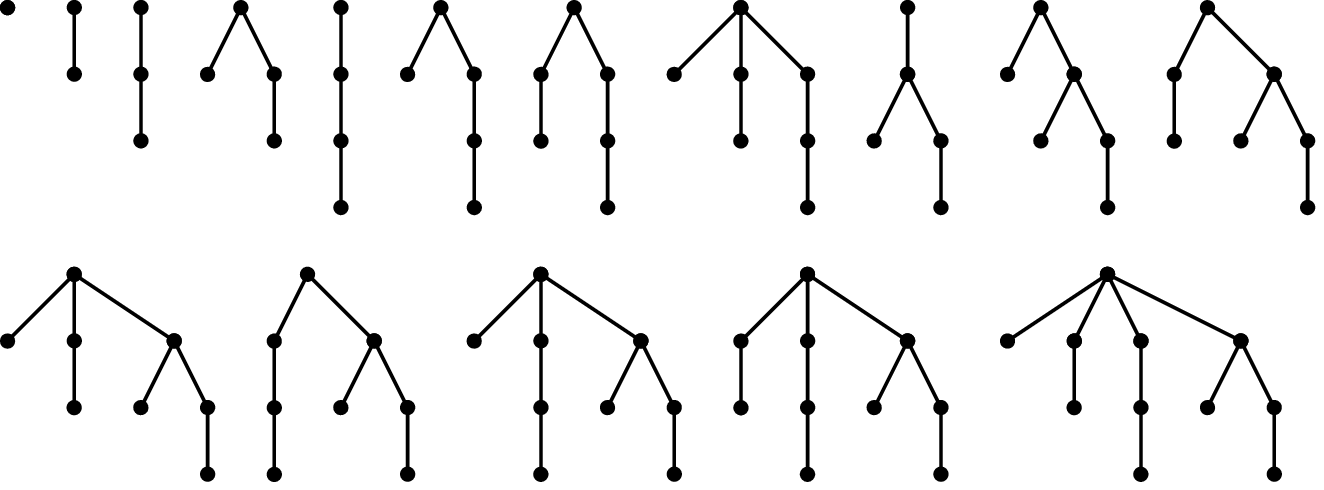}
\caption{Canonical game trees of height $\leq 3$}
\label{arbres hauteur 3}
\end{figure}

We can count the number of canonical game trees of a given height, as on figure \ref{arbres hauteur 3}, which shows the 16 canonical game trees of height $\leq 3$.

\begin{prop}
There are $2^{(2^{(...^{(2^0)})})}$ (with $h+1$ times the number 2) canonical game trees of height $\leq h$.
\end{prop}

This notion of canonical game tree allows us to keep only the necessary and sufficient information of a game tree to describe the possible games from a given position: if two positions have the same canonical game tree then, whatever the kind of rules (normal, mis\`{e}re, or any other rule), the two positions are equivalent and the same player has a winning strategy.

\section{Limit genus}

We state in this section a number of results on the influence of the surface on the game.

\subsection{Case of regions with 3 lives or less}

\begin{prop}
In a position, if a region has 3 lives or less, then whatever the surface associated to this region is, the canonical game tree obtained from that position is the same.
\end{prop}

\begin{proof}
If a region has no or one life, we can't play a move inside it, so that the surface associated has no influence on the game: the proposition is obvious for regions with at most one life.

Let us consider the case of a region with exactly two lives. Whatever the surface associated to this region is, there is only one possible move inside, which is to link the two lives. After this move, there is only one life left in the region, which is the previous case. The moves done outside the region are not influenced by the surface associated to the region, and modifies the region only by decreasing its number of lives. It follows that the proposition is true for regions with at most two lives.

The case of region with three lives is quite similar. Moves outside the region can be treated with the same argument. In the case of a move inside the region, two cases are possible: either the move doesn't create a new region, in which case we are simply back to a region with two lives; or the move split the region in two new ones. But in this case, at least one of the two new regions has two lives, so that another move is possible, and this move is necessarily the last one. It follows that whether we split the region in two or not, the canonical tree is the same.
\end{proof}

\begin{figure}[ht]
\centering
\includegraphics[scale=0.5,angle=-90]{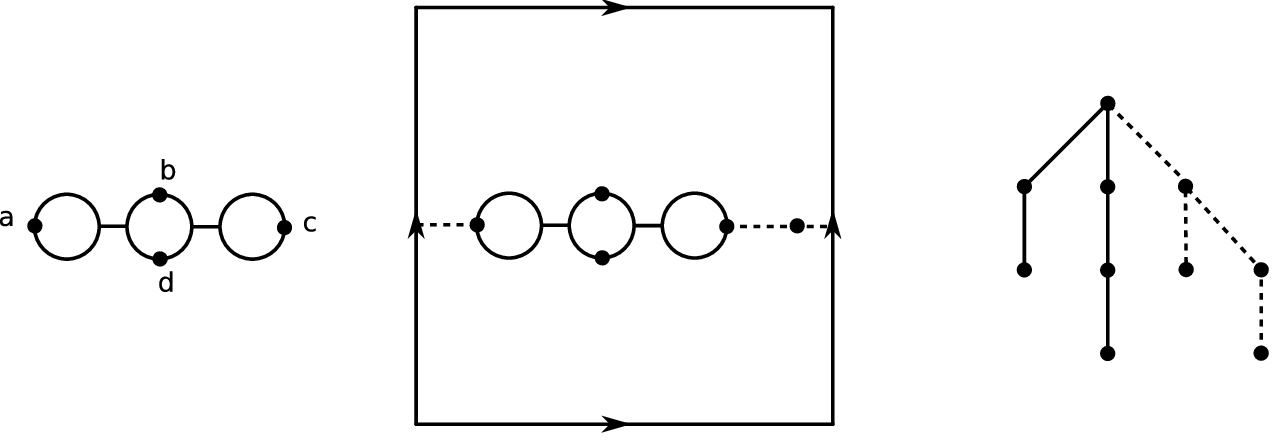}
\caption{Position whose canonical tree depends on the surface}
\label{2A2B}
\end{figure}

Let us remark that 3 is the best value: we have drawn on figure \ref{2A2B} a Sprouts position with 4 lives, whose canonical tree depends on the surface. If the outside surface is a plane, the canonical tree is the one on the right of the figure: either we link a to c, or b to d, and the game ends in two moves. Or we link a to b, b to c, c to d or d to a, and the game ends in three moves. However, if we play on $\T$, there is another possible move, shown with a dotted line, and this move creates the dotted branch in the canonical tree: either we link b to d inside the little circular region, ending the game, or we play any other move, and the game ends with one other move.

\subsection{Limit genus on orientable surfaces}

We can state a more general result on the influence of the surface on the game. Let us begin first with orientable surfaces.

\begin{prop}
For a given region $\mathscr{R}$, there exists an integer $g(\mathscr{R})$, the \emph{limit genus}, such that for every $n\geq g(\mathscr{R})$, and for any position $\mathscr{P}$ including the region $\mathscr{R}$, the canonical tree obtained from $\mathscr{P}$ with the region $\mathscr{R}$ homeomorphic to a surface $\T^{n}$ is the same as the canonical tree obtained from $\mathscr{P}$ with the region $\mathscr{R}$ homeomorphic to a surface $\T^{n+1}$.
\end{prop}

In other words, it means that when there are too many handles in a region $\mathscr{R}$, the game ends necessarily before we can ``use'' all of them. As a consequence, the canonical tree obtained from a position $\mathscr{P}$ remains constant (we get a \emph{limit canonical tree}) from a limit number of handles, which we call the limit genus $g(\mathscr{R})$.

\begin{proof}
We use a proof by induction on the number of lives of the regions. For a given region $\mathscr{R}$ included in a position $\mathscr{P}$, there are four kinds of moves that can change it. Each move changes $\mathscr{R}$ into one or two regions, each having strictly less lives than $\mathscr{R}$, so that the induction hypothesis implies the existence of a limit genus for these resulting regions.
%Todo : improve what follows, the case of a move of kind II.A.(a) is forgotten.
%The genus of $\mathscr{R}$ then needs to be sufficiently great so that whatever move we choose, the genus of the resulting regions is greater than their limit genus. In that way, we are sure that the canonical trees of all children of $\mathscr{P}$ are stable even if the genus of $\mathscr{R}$ increases, and then, that the canonical tree of $\mathscr{P}$ is stable.

The four kinds of move are as follows:
\begin{enumerate}
\item moves of kind I inside the region: these moves lead to regions $\mathscr{R}_i$, with $g(\mathscr{R}_i)=a_i$ (the index $i$ takes a finite number of values because a given position leads only to a finite number of different moves). The surface associated to each $\mathscr{R}_i$ is the same as the surface associated to $\mathscr{R}$. It implies that if $\mathscr{R}$ is homeomorphic to a surface $\T^n$ with $n\geq \max\{a_i\}$, then the canonical tree of a child of $\mathscr{P}$ obtained after playing a move of kind I is the limit canonical tree.
\item moves of kind II.B.1.(a) inside the region: these moves lead to regions $\mathscr{R}_j$, with $g(\mathscr{R}_j)=b_j$. Since this kind of move decreases the genus of the region by one, we need that $n\geq \max\{b_j+1\}$ to be sure of having the limit canonical tree.
\item moves of kind II.A.(a) inside the region: the region is broken into two new regions $\mathscr{R}_k$ and $\mathscr{R}'_k$, with $g(\mathscr{R}_k)=c_k$ and $g(\mathscr{R}'_k)=d_k$. We need this time that $n\geq \max\{c_k+d_k\}$, because this kind of move divides the genus of $\mathscr{R}$ between the two new regions.
\item moves outside the region: we need to consider only the moves that kill one or two lives of the region boundaries (we need to consider all the ways of killing one or two lives, even if the two lives are on different boundaries). These moves lead to regions $\mathscr{R}_l$, with $g(\mathscr{R}_l)=e_l$. Since the surface associated to the region is unchanged, we need only $n\geq \max\{e_l\}$.
\end{enumerate}

We can see now that if we take $n=\max\{a_i;b_j+1;c_k+d_k;e_l\}$, and if the surface associated to $\mathscr{R}$ is $\T^n$, then the canonical game tree obtained from position $\mathscr{P}$ is the limit canonical tree. So $g(\mathscr{R})$ exists, and $g(\mathscr{R})\leq\max\{a_i;b_j+1;c_k+d_k;e_l\}$.
\end{proof}

As an immediate corollary of this proposition, it follows that there is a limit integer $n_0$, such that for every $n\geq n_0$ the game with $p$ starting spots on $\T^n$ has the same winner as the game with $p$ starting spots on $\T^{n_0}$.

\subsection{Case of p=2 on orientable surfaces}

We give here an elementary proof that $n_0=0$ if $p=2$: indeed, the second player has always a winning strategy on an orientable surface. The initial 2 spots have a total of 6 lives, and the winning strategy consists to end the game with two isolated spots (with one life each), so that the game ends in $6-2=4$ moves.

There are 3 possibilities of first move for the first player:
\begin{itemize}
\renewcommand{\labelitemi}{*}
 \item move of kind I: it links the two spots. The second player then plays a move of kind II.A.(a), cutting the game field into two regions. Then, when the first player plays in one region, the second player just plays in the other one and wins.
\item move of kind II.A.(a): the surface is separated into two regions. In this case, the second player links the two spots of the boundary, inside the region without the unused spot. There are only two moves left, with the unused spots, so the second player wins.
\item move of kind II.B.1.(a): it breaks a handle by linking a spot to itself. The second player links the remaining spot to itself, along a parallel path. It breaks the surface into two parts, one of which is a cylinder. Then, if the first player plays in the cylinder, the second player plays in the other part, and conversely, if the first player plays in the other part first, then the second player plays in the cylinder.
\end{itemize}

Figure \ref{solution2} shows the first two moves of this strategy on the torus. The first move is in plain line, and the correct answer of the second player in dotted line.

\begin{figure}[ht]
\centering
\includegraphics[scale=0.5,angle=-90]{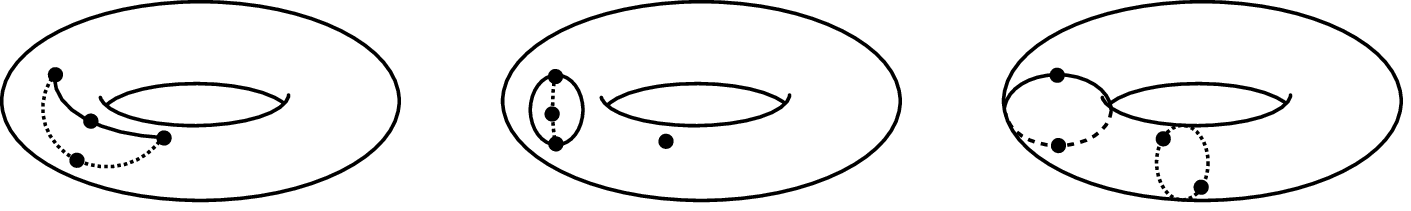}
\caption{Solution of the 2-spot game on the torus}
\label{solution2}
\end{figure}

\subsection{Limit genus on non-orientable surfaces}

The property of limit genus described above also exists on non-orientable surfaces, with a variation: from the limit genus, the canonical tree remains constant on $\P^n$ if and only if $n$ has the same parity.

\begin{prop}
For a given region $\mathscr{R}$, there exists an integer $g'(\mathscr{R})$ such that for every $n\geq g'(\mathscr{R})$, and any position $\mathscr{P}$ including the region $\mathscr{R}$, the canonical tree obtained from $\mathscr{P}$ when $\mathscr{R}$ is homeomorphic to $\P^{n}$ is the same as the canonical tree obtained from $\mathscr{P}$ when $\mathscr{R}$ is homeomorphic to $\P^{n+2}$.
\end{prop}

The difference with the property on orientable surfaces comes from moves of kind II.B.1.(c) and II.B.2.(b), that exist or not depending on the parity of $n$.

In the case of 11 starting spots on $\P^n$, we computed for $1\leq n\leq 8$ that the first player has a winning strategy if and only if $n$ is odd. We can thus conjecture that for this position, the winning player on $\P^n$ depends on the parity of the genus $n$. This property seems very rare, this is the simplest example that we know among thousands of positions.

\section{Results obtained with the program}

\subsection{Orientable surfaces}

We have computed the winning strategy of all starting positions up to 14 spots, with a genus up to 9. In all these cases, the winning player is the same as on the plane. Moreover, there is a stronger pattern with the nimber (see \cite{ww01} for a definition of this concept): the nimber of all these positions is 1 when the first player is winning. We can state the following conjecture, which is more general than the original one of \cite{ajs91}:

\begin{conj}
The nimber of a starting position with $n$ spots on an orientable surface is $0$ if $n=0$, $1$ or $2$ modulo $6$, and $1$ otherwise.
\end{conj}

This result is not really surprising since only moves of kind II.B.1.(a) can change the game on a surface compared to the game on the plane. However, this result can't be deduced from a direct equivalence between plane and orientable surfaces. Indeed, there are a lot of positions whose nimber depends on the surface the game is played on. For example, the canonical tree drawn in solid lines of figure \ref{2A2B} corresponds to a nimber of 2, but with the part drawn in dotted lines, it becomes a nimber of 3. It means that the nimber of this position changes with the surface.

We show on figure \ref{nbo} other positions whose nimber is different on the plane and on another orientable surface. They are simple enough (5 or 6 lives) to be computed by hand. The first one on the left has a nimber of 2 if the outside region is a plane and of 4 for any other orientable surface. The nimbers of the second position are respectively 1 and 4, and 0 and 3 for the two last.

\begin{figure}[ht]
\centering
\includegraphics[scale=0.5]{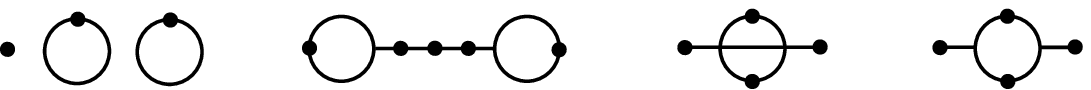}
\caption{Positions whose nimber changes on the torus}
\label{nbo}
\end{figure}

\subsection{Non-orientable surfaces}

The game on non-orientable surfaces shows much more significant differences with the plane. The game with 2 spots gives already a representative example: the second player has a winning strategy on non-orientable surfaces, whereas it was the first player on orientable surfaces.

\begin{figure}[ht]
\centering
\begin{tabular}{ |*{14}{c|} }
\hline
       & 2 & 3 & 4 & 5 & 6 & 7 & 8 & 9 & 10 & 11   & 12 & 13 & 14\tabularnewline
\hline
$\S$   & 0 & 1 & 1 & 1 & 0 & 0 & 0 & 1 & 1  & 1    & 0  & 0  & 0   \tabularnewline
\hline
$\P$   & 1 & 1 & 1 & 0 & 0 & 0 & 0 & 1 & 1  & 1    & 0  & 0  & 0   \tabularnewline
\hline
$\P^2$ & 1 & 1 & 1 & 0 & 0 & 0 & 1 & 1 & 1  & 0    & 0  & 0  & 0   \tabularnewline
\hline
$\P^3$ & 1 & 1 & 1 & 0 & 0 & 0 & 1 & 1 & 1  & $>3$ & 0  & 0  & $>2$\tabularnewline
\hline
$\P^4$ & 1 & 1 & 1 & 0 & 0 & 0 & 1 & 1 & 1  & 0    & 0  & 0  & $>0$\tabularnewline
\hline
$\P^5$ & 1 & 1 & 1 & 0 & 0 & 0 & 1 & 1 & 1  & $>1$ & 0  & 0  & $>2$\tabularnewline
\hline
$\P^6$ & 1 & 1 & 1 & 0 & 0 & 0 & 1 & 1 & 1  & 0    & 0  & 0  & $>0$\tabularnewline
\hline
$\P^7$ & 1 & 1 & 1 & 0 & 0 & 0 & 1 & 1 & 1  & $>1$ & 0  & 0  & $>2$\tabularnewline
\hline
$\P^8$ & 1 & 1 & 1 & 0 & 0 & 0 & 1 & 1 & 1  & 0    & 0  & 0  & $>0$\tabularnewline
\hline
\end{tabular}
\caption{Nimbers of starting positions with $p$ spots on $\P^n$}
\label{tablo}
\end{figure}

Figure \ref{tablo} shows the results we obtained with our program, for all starting positions from 2 to 14 spots, on $\P^n$ surfaces with genus $n$ between 0 and 8.

The first surprising result is the difference between the 2-spot game on $\S$, where the second player wins, and on $\P^n$ ($n\geq 1$), where the first player wins. For the 5-spot game, this is the exact contrary. And for the 8-spot game, the second player wins on $\S$ and $\P$, whereas the first player wins on $\P^n$ ($n\geq 2$).

Moreover, some starting positions have a nimber different from 0 and 1, the simplest of them being the positions with 11 starting spots on $\P^3$. However, the position is rather complicated, and we lacked time to compute its nimber. We could only compute that it is at least 4. 

As stated before, we can also notice that the winning player alternates in the case of 11 starting spots.

But contrary to the game on orientable surfaces, it seems difficult to conjecture any global pattern. Our current results are too partial to show any regularity, it it exists at all.

\subsection{Misere game}

We also computed the outcome of the mis\`ere version of the $n$-spot game, for $n\leq 11$. These results are given in the table \ref{tablo2}.

\begin{figure}[ht]
\centering
\begin{tabular}{ |*{14}{c|} }
\hline
       & 2 & 3 & 4 & 5 & 6 & 7 & 8 & 9 & 10 & 11 \tabularnewline
\hline
$\T^3$ & L & L & L & W & W & L & L & L & W  & W  \tabularnewline
\hline
$\T^2$ & L & L & L & W & W & L & L & L & W  & W  \tabularnewline
\hline
$\T$   & L & L & L & W & W & L & L & L & W  & W  \tabularnewline
\hline
$\S$   & L & L & L & W & W & L & L & L & W  & W  \tabularnewline
\hline
$\P$   & L & L & W & W & W & L & L & L & W  & W  \tabularnewline
\hline
$\P^2$ & L & L & W & W & W & L & L & W & W  & W  \tabularnewline
\hline
$\P^3$ & L & L & W & W & W & L & L & W & W  & W  \tabularnewline
\hline
$\P^4$ & L & L & W & W & W & L & L & W & W  & W  \tabularnewline
\hline
\end{tabular}
\caption{Outcome of mis\`ere starting positions}
\label{tablo2}
\end{figure}

\section{Conclusion}

The surprisingly regular pattern indicating the winning player on the plane seems to be extendable to orientable surfaces: the influence of the surface structure on the game seems too little to change the winning player.

On the contrary, theory as well as programming are much more complicated on non-orientable surfaces, leading to notable changes in the winning strategies. For a given starting position, the winning player is not always the same, and an interesting phenomenon appears on some positions: the winner can depend on the parity of the surface genus.

We have also proved that, for a given position, it is sufficient to solve the game up to a certain limit genus, in order to deduce which player has a winning strategy on any surface with a greater genus. For example, we solved by hand the case of the starting position with 2 spots on any orientable surface. This theory of the limit genus proves that studying completely a given position on any compact surface can be reduced to the study of a finite number of cases. An interesting extension of the program would be to compute the value--or an upper value--of this limit genus, so that we can state results true on any compact surface, as we did with the 2 spots case.

The program we used for computation is available with its source code on our web site http://sprouts.tuxfamily.org/ under a GNU license.

\subsection*{Thanks}

We would like to thank Jean-Francois Peltier and ``Ypercube'', who have described independently the kind of possible moves on compact surfaces, and who found a mistake in our own first description. We thank Jean-Paul Delahaye, who gave us the opportunity to continue our work on Sprouts. We also thank him and Dan Hoey for their corrections on this article.

\bibliography{surfaces}
\bibliographystyle{amsplain}

\end{document}